\documentclass[12pt]{article}
\input{amssym}
\newtheorem{obs}{Observation}[section]
\newtheorem{thm}[obs]{Theorem}

\newtheorem{cor}[obs]{Corollary}
\newtheorem{lemma}[obs]{Lemma}

\newtheorem{pro}[obs]{Proposition}
\newtheorem{lem}[obs]{Theorem}

\newtheorem{preproof}{{\bf Proof.}}

\newenvironment{proof}[1]{\begin{preproof}{\rm
               #1}\hfill{$\rule{2mm}{2mm}$}}{\end{preproof}}

\def\newpic#1{}
\date{}
\begin{document}
\title{
{\Large{\bf Metric dimension of lexicographic product of some known
graphs }}}
%
{\small
\author{Mohsen Jannesari
\\
[1mm]
{\small \it  Department of  Science}\\
{\small \it  Shahreza Campus, University of Isfahan, Iran} \\
{\small \it Email: m.jannesari@shr.ui.ac.ir}
}}
\maketitle \baselineskip15truept
\begin{abstract}
For an ordered set $W=\{w_1,w_2,\ldots,w_k\}$ of vertices and a
vertex $v$ in a connected graph $G$, the ordered $k$-vector
$r(v|W):=(d(v,w_1),d(v,w_2),\ldots,d(v,w_k))$ is  called  the
(metric) representation of $v$ with respect to $W$, where $d(x,y)$
is the distance between the vertices $x$ and $y$. The set $W$ is
called  a resolving set for $G$ if distinct vertices of $G$ have
distinct representations with respect to $W$. The minimum
cardinality of a resolving set for $G$ is its metric dimension. In
this paper, we investigate the metric dimension of the lexicographic
product  of graphs $G$ and $H$, $G[H]$ for some known graphs.
\end{abstract}

\section{Introduction}
In this section, we present some definitions and known results
which are necessary  to prove our main results. Throughout this
paper, $G$ is a  simple  graph
with vertex set $V(G)$, edge set $E(G)$ and order $n(G)$. We use $\overline G$
for the complement of  graph  $G$.
The distance between two
vertices $u$ and $v$, denoted by $d_G(u,v)$, is the length of a
shortest path between $u$ and $v$ in $G$. Also, $N_G(v)$ is the
set of all neighbors of vertex $v$ in $G$. We write these simply
$d(u,v)$ and $N(v)$,  when no confusion can arise.  The diameter of a connected graph $G$ is
${\rm diam}(G)=\max_{_{u,v\in V(G)}} d(u,v)$.
The
symbols $(v_1,v_2,\ldots, v_n)$ and $(v_1,v_2,\ldots,v_n,v_1)$
represent a path of order $n$, $P_n$, and a cycle of order $n$,
$C_n$, respectively. We also use notation $\bf1$ for the vector $(1,1,\ldots,1)$ and $ \bf2$ for $(2,2,\ldots,2)$.
\par
 For an ordered subset $W=\{w_1,\ldots,w_k\}$ of $V(G)$ and a
vertex $v$ of a connected graph $G$, the {\it metric representation}  of $v$ with
respect to $W$ is
$$r(v|W)=(d(v,w_1),\ldots,d(v,w_k)).$$
The set $W$ is  a {\it resolving set} for
$G$ if the distinct vertices of $G$ have different metric representations, with
respect to $W$.
A resolving set $W$ for $G$ with
minimum cardinality is  a {\it metric basis} of $G$, and its
cardinality is the {\it metric dimension} of $G$, denoted by
$\dim(G)$.
\par The concepts of resolving
sets and metric
dimension of a graph
 were introduced independently by Slater~\cite{Slater1975}
and by Harary and Melter~\cite{Harary}.
 Resolving sets have several applications in diverse areas such as  coin weighing problems~\cite{coin}, network
discovery and verification~\cite{net2}, robot
navigation~\cite{landmarks}, mastermind game~\cite{cartesian
product}, problems of pattern recognition and image
processing~\cite{digital}, and combinatorial search and
optimization~\cite{coin}.
For more results about  resolving sets and metric dimension see~\cite{baily,K dimensional,cartesian product,Ollerman,extermal}.
\par The {\it lexicographic product} of graphs $G$ and $H$,  denoted by
$G[H]$, is a graph with vertex set  $V(G)\times
V(H):=\{(v,u)~|~v\in V(G),u\in V(H)\}$, where  two vertices
$(v,u)$ and $(v',u')$ are  adjacent whenever,  $v$ is adjacent to $ v'$, or
$v=v'$ and $u$ is adjacent to $u'$. When the order of $G$ is at least $2$, it
is easy to see that $G[H]$ is a connected graph if and only if
$G$ is a connected graph.
\par Jannesari and Omoomi~\cite{lexico} studied the metric dimension of the lexicographic product of graphs using {\it adjacency
dimension} of graphs.
\par
Let $G$ be a graph, and $W=\{w_1,\ldots,w_k\}\subseteq V(G)$.
For each vertex $v\in V(G)$, the {\it adjacency
representation} of $v$ with respect to $W$ is the $k$-vector
$$r_2(v|W)=(a_G(v,w_1),\ldots,a_G(v,w_k)),$$ where $a_G(v,w_i)=min\{2,d_G(v,w_i)\};1\leq i\leq k.$
 The set $W$ is an {\it adjacency  resolving set} for $G$
if the vectors $ r_2(v|W)$ for $v\in V(G)$ are distinct.
The minimum cardinality of an adjacency resolving set
is the adjacency dimension of $G$, denoted by
$\dim_2(G)$. An adjacency resolving set of cardinality
$\dim_2(G)$ is  an {\it adjacency basis} of $G$.
\par
We say that a set $W$ {\it  (adjacency) resolves} a set $T$ of vertices
in $G$ if the  (adjacency) metric representations of vertices in $T$ with respect to W are distinct.
  To determine whether a given set $W$ is a (an adjacency) resolving
set for $G$, it is sufficient to look at the  (adjacency) metric representations of
vertices in $V(G)\backslash W$, because $w\in W$ is the unique
vertex of $G$ for which $d(w,w)=0$.
\par In this paper we investigate  metric dimension of the lexicographic product, for some known graphs.
 Jannesari and Omoomi~\cite{lexico} studied the metric dimension of the lexicographic product of graphs.
To express their results we need some definitions.
\par Two distinct vertices $u,v$ are said  {\it twins}
if $N(v)\backslash\{u\}=N(u)\backslash\{v\}$. It is called that
$u\equiv v$ if and only if $u=v$ or $u,v$ are twins.
In~\cite{extermal}, it is proved that ``$\equiv$" is an equivalent
relation. The equivalence class of a vertex $v$ is denoted by
$v^*$. Hernando et al.~\cite{extermal} proved that $v^*$ is a
clique or an independent set in $G$. As in ~\cite{extermal}, we
say $v^*$ is of type (1), (K), or (N) if $v^*$ is a class of size
$1$, a clique of size at least $2$, or an independent set of size
at least $2$. We denote the number of equivalence classes of $G$
with respect to ``$\equiv$" by $\iota(G)$. We mean by
$\iota_{_K}(G)$ and $\iota_{_N}(G)$, the number of classes of
type (K) and type (N) in $G$, respectively. We also use $a(G)$ and
$b(G)$ for the number of all vertices in  $G$ which have at least
an adjacent twin and a none-adjacent twin vertex in $G$,
respectively.  On the other way, $a(G)$ is the number of all
vertices in the classes of type (K) and $b(G)$ is the number of
all vertices in the classes of type (N). Clearly,
$\iota(G)=n(G)-a(G)-b(G)+\iota_{_N}(G)+\iota_{_K}(G)$.

\begin{obs}~{\rm\cite{extermal}}\label{twins}
Suppose that $u,v$ are twins in a  graph $G$ and $W$ resolves $G$.
Then $u$ or $v$ is in $W$. Moreover, if $u\in W$ and $v\notin W$,
then $(W\setminus\{u\})\cup \{v\}$ also resolves $G$.
\end{obs}
\begin{lem}~\rm\cite{Ollerman}\label{B=1,B=n-1}
Let $G$ be a connected graph of order $n$.
Then,\begin{description}\item (i) $\dim(G)=1$ if and only if
$G=P_n$,\item (ii) $\dim(G)=n-1$ if and only if $G=K_n$.
\end{description}
\end{lem}


\begin{pro}\label{B(H)<B2(H)}\rm\cite{lexico} For every connected graph $G$,   $\dim(G)\leq \dim_2(G)$.
\end{pro}
\begin{pro}\label{B2(H)=B2(overlineH)}\rm\cite{lexico}
For every graph $G$,  $\dim_2(G)=\dim_2(\overline{G})$.
\end{pro}
 Let $G$ be a graph of order
$n$.  It is easy to see that, $1\leq \dim_2(G)\leq n-1$. In the
following proposition,  all graphs $G$ with
$\dim_2(G)=1$ and all graphs $G$ of order $n$ and
$\dim_2(G)=n-1$ are  characterized.
\begin{pro}\label{characterization 1, m-1}\rm\cite{lexico}
If $G$ is a graph of order $n$, then
\begin{description}
 \item (i) $\dim_2(G)=1$ if and only if
$G\in\{P_1,P_2,P_3,\overline{P}_2,\overline{P}_3\}$.
\item (ii) $\dim_2(G)=n-1$ if and only if $G=K_n$ or
$G=\overline{K}_n$.
\end{description}
\end{pro}
\begin{pro}\label{B_2(P_m),B_2(C_m)}\rm\cite{lexico}
If $n\geq 4$, then $\dim_2(C_n)=\dim_2(P_n)=\lfloor{{2n+2}\over
5}\rfloor$.
\end{pro}

\begin{pro}\label{B_2 multipartite}\rm\cite{lexico} If $K_{m_1,m_2,\ldots,m_t}$ is
the complete $t$-partite graph, then
$$\dim_2(K_{m_1,m_2,\ldots,m_t})=\dim(K_{m_1,m_2,\ldots,m_t})=\left\{
\begin{array}{ll}
m-r-1 &  ~{\rm if}~r\neq t, \\
m-r &  ~{\rm if}~r=t, \end{array}\right.$$ where $m_1,
m_2,\ldots,m_r$ are at least $2$, $m_{r+1}=\cdots=m_t=1$, and
$\sum_{i=1}^tm_i=m$.
\end{pro}
\par Jannesari and Omoomi obtained the metric dimension of lexicographic product of graphs through the following four  theorems.
\begin{thm}\label{B1 B2}\rm\cite{lexico}
Let $G$ be a connected graph of order $n$ and $H$ be an arbitrary
graph. If there exist two adjacency bases $W_1$ and $W_2$ of $H$
such that, there is no vertex with adjacency representation
$\bf1$ with respect to $W_1$ and no vertex with adjacency
representation $\bf2$ with respect to $W_2$, then
$\dim(G[H])=\dim(G[\overline H])=n\dim_2(H)$.
\end{thm}
\begin{thm}\label{thm generalG[H]}\rm\cite{lexico}
Let $G$ be a connected graph of order $n$ and $H$ be an arbitrary
graph. If for each adjacency basis $W$ of $H$ there exist vertices
with adjacency representations $\bf1$ and  $\bf2$ with
respect to $W$, then $\dim(G[H])=\dim(G[\overline
H])=n(\dim_2(H)+1)-\iota(G).$
\end{thm}
\begin{thm}\label{nB2+a(G) K(G)}\rm\cite{lexico}
Let $G$ be a connected graph of order $n$ and $H$ be an arbitrary
graph. If $H$ has the following properties
\begin{description}\item (i) for each adjacency basis of $H$
there exists a vertex with adjacency representation $\bf1$,
\item (ii) there exists an adjacency basis $W$ of $H$ such that
there is no vertex with adjacency representation $\bf2$ with
respect to $W$,
\end{description} then
$\dim(G[H])=n\dim_2(H)+a(G)-\iota_{_K}(G).$
\end{thm}
\begin{thm}\label{nB2+b(G) N(G)}\rm\cite{lexico}
Let $G$ be a connected graph of order $n$ and $H$ be an arbitrary
graph. If $H$ has the following properties
\begin{description}\item (i) for each adjacency basis of $H$
there exists a vertex with adjacency representation  $\bf2$,
\item (ii) there exists an adjacency basis $W$ of $H$ such that
there is no vertex with adjacency representation $\bf1$ with
respect to $W$,
\end{description} then
$\dim(G[H])=n\dim_2(H)+b(G)-\iota_{_N}(G).$
\end{thm}

\begin{cor}\label{no twin}\rm\cite{lexico} If $G$  has no  pair  of twin
vertices, then $\dim (G[H])=n\dim_2(H)$.
\end{cor}
\section{Main results}\label{results}
In this section we investigate metric dimension of the lexicographic product of graphs for some families of graphs.
Theorems~\ref{B1 B2}, \ref{thm generalG[H]}, \ref{nB2+a(G) K(G)} and \ref{nB2+b(G) N(G)} imply that to find the exact value of
$\dim(G[H])$, we need to know all twin vertices in $G$ and adjacency resolving sets for $H$. 

By Corollary~\ref{no twin} to compute the $\dim(G[H])$, where $G$
has no any pair of twin vertices,  it is enough to obtain the
value of $\dim_2(H)$.

\begin{lemma}\label{keneser does not have twins}
If $KG(k,r)$, $k\geq 2r+1$ be the Kneser graph, then $G$  have no
any pair of twin vertices.
\end{lemma}
\begin{proof}{ If $A$ and $B$ are distinct twin vertices in $G$, then $A\cap
C=\emptyset$ if and only if $B\cap C=\emptyset$, for each $C\in
V(G)$. Now let $C\in V(G)\backslash\{A,B\}$, $A\cap C=\emptyset$,
$x\in A\backslash B$, and $y\in C$. Let $D=C\cup
\{x\}\backslash\{y\}$. Therefore, $A\cap D\neq \emptyset$ and
$B\cap D=\emptyset$, which is a contradiction. }\end{proof}

 Note that the line graph  of  the complete graph $K_n$
is the complement of $KG(n,2)$. Since all twin vertices of a graph
are twins in its complement, as well;  by Lemma~\ref{keneser does
not have twins}, $L(K_n)$, $n\geq5$,   have  no any pair of twin
vertices. Also,  the path $P_n$,  $\overline P_n$, $n\geq 4$, and
the cycle $C_n$,  $\overline C_n$, $n\geq 5$,   have  no any pair
of twin vertices. Thus, by Theorems~\ref{thm generalG[H]}
 the exact values when  $G\in\{\overline
P_n~(n\geq4),\overline C_n(n\geq5),L(K_n)~(n\geq5),K(k,r)\}$ and
 $H\in\{P_m,C_m,\overline P_m,\overline C_m,K_m,\overline
K_m,P,K_{m_1,\ldots,m_t}\}$ are obtained.

To study the adjacency basis of a graph $H$, we need the following
definitions. Let $S$ be a subset of  vertices of $H$, where
$|s|\geq 2$. The set of vertices of a nonempty connected component
of the induced subgraph by $V(H)\backslash S$ of $H$ is called a
{\it gap} of $H$. This definition agrees with the one in~\cite{K
dimensional} which is given  for  the cycle $C_m$.  If $Q_1,Q_2$
are two gaps of $S$ for which there exists a vertex $x\in S$ such
that the induced subgraph by $Q_1\cup Q_2\cup\{x\}$ is connected,
then $Q_1$ and $Q_2$ are called {\it neighboring gaps}.
In~\cite{K dimensional}, the following observation is expressed
for the gaps of a basis of $C_m\vee K_1$. Particularly,  it is
true for an adjacency basis of $C_m$.
\begin{obs}\label{obs C_m}
 If $B$ is an adjacency basis of $C_m$,
then
\begin{description}
\item (1) Every gap of $B$ contains at most three vertices.
 \item  (2) At most one gap of $B$ contains three vertices.
 \item (3) If a gap  of $B$ contains at least two vertices, then any neighboring
gaps of which contain one vertex.
\end{description}
\end{obs}
It is easy to see the following observation for $P_m$.
\begin{obs}\label{obs P_m}
 Let $B$ be an adjacency basis of $P_m=(w_1,w_2,\ldots,w_m)$. If $R_1$
 and $R_2$ are  gaps of $P_n$ with $w_1\in R_1$ and
 $w_m\in R_2$,
then
\begin{description}
\item (1) Every gap of $B$ contains at most three vertices and
$|R_i|\leq2$, where $1\leq i\leq2$.
 \item (2) At most one gap of
$B$ contains three vertices and at most one of the gaps $R_1$ and
$R_2$  contains two vertices.
 \item (3) If $|R_i|=2$ for some
$i,~1\leq i\leq2$, then all gaps of $B$  contains at most two
vertices.
\item (4) If a gap of $B$ contains at least two
vertices, then any neighboring gaps of which is neither $R_1$ nor
$R_2$ and contain one vertex.
\end{description}

\end{obs}
\begin{pro}\label{B(G[pm,cm])}
Let $G$ be a connected graph of order $n$ and $H\in\{P_m,C_m\}$,
$m=5k+r$, where $m\notin\{2,3\}$.
\begin{description}
 \item(i) If $r$ is  even, then
$\dim(G[H])=\dim(G[\overline H])=n\lfloor{2m+2\over5}\rfloor$.
\item (ii) If $m=6$, then $\dim(G[H])=\dim(G[\overline
H])=n\lfloor{2m+2\over5}\rfloor+a(G)+b(G)-\iota_{_K}(G)-\iota_{_N}(G)$.
\item
(iii) If  $r$ is  odd and $m\neq6$, then
$\dim(G[H])=n\lfloor{2m+2\over5}\rfloor+b(G)-\iota_{_N}(G)$ and
$\dim(G[\overline
H])=n\lfloor{2m+2\over5}\rfloor+a(G)-\iota_{_K}(G)$.
\end{description}
\end{pro}
 \begin{proof}{Let $P_m=(w_1,w_2,\ldots,w_m)$ and $C_m=(w_1,w_2,\ldots,w_m,w_1)$.
 If $m=4$, then the set $B_4=\{w_1,w_4\}\subseteq V(H)$ is an
 adjacency basis of $H$ and $r_2(w_i|B_4)$  is neither  $\bf1$ nor
$\bf2$, for each $i$, $1\leq i\leq 4$. Therefore, by
Proposition~\ref{B_2(P_m),B_2(C_m)} and Theorem~\ref{B1 B2},
$\dim(G[H])=\dim(G[\overline H])=n\lfloor{2m+2\over5}\rfloor$.
If $m=5$, then the sets $B_1=\{w_1,w_5\}$ and $B_1=\{w_2,w_4\}$
are adjacency bases of $H$ and for each $i$, $1\leq i\leq5$,
$r_2(w_i|B_1)$ is not  $\bf1$ and $r_2(w_i|B_2)$ is not
 $\bf2$. Hence, by Lemma~\ref{B_2(P_m),B_2(C_m)} and
Theorem~\ref{B1 B2}, $\dim(G[H])=\dim(G[\overline
H])=n\lfloor{2m+2\over5}\rfloor$. If $m=6$, then it is easy to
check that for each adjacency basis $A$ of $H$ there exist
vertices $x_{_A},y_{_A}\in V(H)$ such that $x_{_A}\sim w$ and
$y_{_A}\nsim w$ for each $w\in A$. Consequently, by
Theorem~\ref{thm generalG[H]}, $\dim(G[H])=\dim(G[\overline
H])=n\lfloor{2m+2\over5}\rfloor+a(G)+b(G)-\iota_{_K}(G)-\iota_{_N}(G)$.
\par Now, let $m\geq 7$. By Proposition~\ref{B_2(P_m),B_2(C_m)},
$\dim_2(H)\geq 3$. Let $B$ be  an adjacency basis of $H$. Since
each vertex of $H$ has at most two neighbors, $r_2(w|B)$ is not
 $\bf1$, for each $w\in V(H)$. If $r$ is even, then, let
$S_0=\{w_{5q+2},w_{5q+4}|~0\leq q\leq k-1\}$,
$S_2=S\cup\{w_{5k+1}\}$, and $S_4=S\cup\{w_{5k+1},w_{5k+3}\}$.
Thus, the set $S_t$, is an adjacency basis of $H$ when $r=t$,
$t\in\{0,2,4\}$. Also, $r_2(w|S_t)$ is neither $\bf1$ nor
 $\bf2$, for each $w\in V(H)$ and $t\in\{0,2,4\}$. Hence, by
by Lemma~\ref{B_2(P_m),B_2(C_m)} and Theorem~\ref{B1 B2},
$\dim(G[H])=\dim(G[\overline H])=n\lfloor{2m+2\over5}\rfloor$.
\par If $r$ is odd, then Observations~\ref{obs C_m} and \ref{obs
P_m} imply that for each adjacency basis $A$ of $H$ there exists a
vertex $y_{_A}\in V(H)$ such that $y_{_A}\nsim w$ for each $w\in
A$. Therefore, by Theorem~\ref{nB2+b(G) N(G)},
$\dim(G[H])=n\lfloor{2m+2\over5}\rfloor+b(G)-\iota_{_N}(G)$.
Since the adjacency bases of $H$ and $\overline H$ are the same,
for each adjacency basis $Q$ of $\overline H$ there exists a
vertex $x_{_Q}\in V(\overline H)$ such that $x_{_Q}\sim u$ for
each $u\in Q$. Hence, by Theorem~\ref{nB2+a(G) K(G)},
$\dim(G[\overline
H])=n\lfloor{2m+2\over5}\rfloor+a(G)-\iota_{_K}(G)$. }\end{proof}
\begin{cor}\label{B(K_n[p_m,c_m]} Let $m=5k+r$. If $H\in\{P_m,C_m\}$, then for
all $n\geq 2$, \begin{description}\item (1)
$\dim({K_n[H]})=\left\{
\begin{array}{ll}
2n-1 &  ~{\rm if}~H=P_2~{\rm or}~H=P_3, \\
3n-1 &  ~{\rm if}~H\in\{C_3,P_6,C_6\},\\
n\lfloor{2m+2\over5}\rfloor &   ~{\rm otherwise}.
\end{array}\right.$
\item (2) $\dim({P_n[H]})=\left\{
\begin{array}{ll}
5 &  ~{\rm if}~n=2~{\rm and}~H=C_3, \\
2n &  ~{\rm if}~n\neq2~{\rm and}~H=C_3,\\
n\lfloor{2m+2\over5}\rfloor+1 &  ~{\rm if}~n=2~{\rm and}~H\in\{P_2,P_3,P_6,C_6\},\\
n\lfloor{2m+2\over5}\rfloor+1 &  ~{\rm if}~n=3,~r~{\rm is ~odd,}~{\rm and}~H\neq~C_3,\\
n\lfloor{2m+2\over5}\rfloor &   ~{\rm otherwise}.
\end{array}\right.$
\item (3) $\dim({C_n[H]})=\left\{
\begin{array}{ll}
8 &  ~{\rm if}~n=3~{\rm and}~H=C_3, \\
2n &  ~{\rm if}~n\neq3~{\rm and}~H=C_3,\\
n\lfloor{2m+2\over5}\rfloor+2 &  ~{\rm if}~n=3~{\rm and}~H\in\{P_2,P_3,P_6,C_6\},\\
n\lfloor{2m+2\over5}\rfloor+2 &  ~{\rm if}~n=4,~r~{\rm is ~odd,}~{\rm and}~H\neq~C_3,\\
n\lfloor{2m+2\over5}\rfloor &   ~{\rm otherwise}.
\end{array}\right.$
\item (4) $\dim({K_{n_1,n_2,\ldots,n_t}[H]})=\left\{
\begin{array}{ll}
n\lfloor{2m+2\over5}\rfloor+t-j-1 &  ~{\rm if}~H=P_2~{\rm and}~j\neq t, \\
n(m-1)+t-j-1 &  ~{\rm if}~H=C_3~{\rm and}~j\neq t,\\
n(m-1) &  ~{\rm if}~H=C_3~{\rm and}~j=t,\\
n\lfloor{2m+2\over5}\rfloor+n-j-1 &   ~{\rm if}~H\in\{P_3,P_6,C_6\}~{\rm and}~j\neq t,\\
n\lfloor{2m+2\over5}\rfloor+n-t &   ~{\rm if}~H\in\{P_3,P_6,C_6\}~{\rm and}~j=t,\\
n\lfloor{2m+2\over5}\rfloor+n-t &   ~{\rm if}~m\geq7~{\rm and}~r~{\rm is~odd},\\
n\lfloor{2m+2\over5}\rfloor &   ~{\rm otherwise}.
\end{array}\right.$
\end{description}where $n_1, n_2,\ldots,n_j$ are at least $2$, $n_{j+1}=\ldots=n_t=1$, and
$\sum_{i=1}^tn_i=n$.
\end{cor}
\begin{proof}{ Since $K_n$ does not have any pair of none-adjacent
twin vertices, by Proposition~\ref{B(G[pm,cm])},
$\dim(K_n[H])=n\lfloor{2m+2\over5}\rfloor$ for
$m\notin\{2,3,6\}$. If $H=P_2$ or $H=C_3$, then $K_n[H]$ is the
complete graph and hence, $\dim(K_n[P_2])=2n-1$ and
$\dim(K_n[C_3])=3n-1$.
\par Now let $H\in\{P_3,P_6,C_6\}$. Also,
let $P_m=(w_1,w_2,\ldots,w_m)$, $C_m=(w_1,\ldots,w_m,w_1)$, and $B$ is a
basis of $K_n[H]$. By the proof of Lemma~\ref{lower bound}, $B$
contains at least $\dim_2(H)$ vertices from each set
$R_i=\{v_{rs}\in V(K_n[H])|r=i\}$, and $B\cap R_i$ resolves $R_i$,
$1\leq i\leq n$. Let $J=\{i|~\dim_2(H)=|B\cap R_i|\}$. If
$|J|\geq2$, then there exist $i,j$, $1\leq i,j\leq n$, such that
$|B\cap R_i|=|B\cap R_j|=\dim_2(H)$. Let $A_i=\{w_s|v_{is}\in
B\cap R_i\}$ and $A_j=\{w_s|v_{js}\in B\cap R_j\}$. Since
 $d_{K_n[H]}(v_{rs},v_{rq})=a_H(w_s,w_q)$ for each
 $r,s,q,~1\leq r\leq n,~1\leq s,q\leq m$, we conclude that $A_i$ and $A_j$
 are adjacency bases of $H$. On the other hand, for each adjacency
 basis $A$ of $H$ there exist a vertex $w\in V(H)$ such that
 $r_2(w|A)=(1,1,\ldots,1)$. Therefore, there exist vertices $w_1,
 w_2\in V(H)$ such that $r_2(w_1|A_i)=r_2(w_2|A_j)=(1,1,\ldots,1)$.
 Consequently, $r(v_{i1}|B\cap R_i)=r(v_{j2}|B\cap
 R_j)=(1,1,\ldots,1)$. Also, we have $r(v_{i1}|B\backslash R_i)=r(v_{j2}|B\backslash
 R_j)=(1,1,\ldots,1)$. Hence, $r(v_{i1}|B)=r(v_{j2}|B)$, which is
 a contradiction. Thus, $|J|\leq1$. Therefore, $\dim(K_n[H])\geq
 n\dim_2(H)+n-1$. On the other hand, the set $\{v_{rs}\in
 V(K_n[P_3])|s\neq3\}\backslash\{v_{12}\}$ is a resolving set for
 $K_n[P_3]$ with cardinality $n\dim_2(H)+n-1=2n-1$. Also, the set
 $\{v_{rs}\in
 V(K_n[H])|2\leq s\leq4\}\backslash\{v_{13}\}$ is a resolving set for
 $K_n[H]$, for $H\in\{P_6,C_6\}$. Consequently,
 $\dim(K_n[P_6])=\dim(K_n[C_6])=3n-1$.
 }\end{proof}
\begin{cor}\label{B(K_n[p_m,c_m]} Let $m=5k+r$. If $H\in\{\overline P_m,\overline C_m\}$, then for
all $n\geq 2$, \begin{description}\item (1)
$\dim({K_n[H]})=\left\{
\begin{array}{ll}
n\lfloor{2m+2\over5}\rfloor+n-1 &  ~{\rm if}~H\neq \overline C_3~{\rm and}~r~{\rm is~odd}, \\
2n &  ~{\rm if}~H=\overline C_3,\\
n\lfloor{2m+2\over5}\rfloor &   ~{\rm otherwise}.
\end{array}\right.$
\item (2) $\dim({P_n[H]})=\left\{
\begin{array}{ll}
4 &  ~{\rm if}~n=2~{\rm and}~H=\overline C_3, \\
n\lfloor{2m+2\over5}\rfloor+1 &  ~{\rm if}~n=2,~r~{\rm is ~odd,~ and}~H\neq \overline C_3,\\
n\lfloor{2m+2\over5}\rfloor+1 &  ~{\rm if}~n=3~{\rm and}
~H\in\{\overline P_2,\overline P_3,\overline P_6,\overline C_6\},\\
7 &  ~{\rm if}~n=3,~{\rm and}~H=\overline C_3,\\
n(m-1) &  ~{\rm if}~n\geq4~{\rm and}~H=\overline C_3,\\
n\lfloor{2m+2\over5}\rfloor &   ~{\rm otherwise}.
\end{array}\right.$
\item (3) $\dim({C_n[H]})=\left\{
\begin{array}{ll}
6 &  ~{\rm if}~n=3~{\rm and}~H=\overline C_3, \\
n\lfloor{2m+2\over5}\rfloor+2 &  ~{\rm if}~n=3~{\rm and}~H\neq\overline C_3,\\
n\lfloor{2m+2\over5}\rfloor+2 &  ~{\rm if}~n=4~{\rm and}~H\in\{\overline P_2,\overline P_3,\overline P_6,\overline C_6\},\\
10 &  ~{\rm if}~n=4,~{\rm and}~H=~\overline C_3,\\
n(m-1) &  ~{\rm if}~n\geq5,~{\rm and}~H=~\overline C_3,\\
n\lfloor{2m+2\over5}\rfloor &   ~{\rm otherwise}.
\end{array}\right.$
\item (4) $\dim({K_{n_1,n_2,\ldots,n_t}[H]})=\left\{
\begin{array}{ll}
n\lfloor{2m+2\over5}\rfloor+n-t &  ~{\rm if}~H=\overline P_2, \\
n(m-1)+n-t &  ~{\rm if}~H=\overline C_3,\\
n\lfloor{2m+2\over5}\rfloor+n-j-1 &   ~{\rm if}~H\in\{\overline P_3,\overline P_6,\overline C_6\}~{\rm and}~j\neq t,\\
n\lfloor{2m+2\over5}\rfloor+n-t &   ~{\rm if}~H\in\{\overline P_3,\overline P_6,\overline C_6\}~{\rm and}~j=t,\\
n\lfloor{2m+2\over5}\rfloor+t-j-1 &   ~{\rm if}~m\geq7,~r~{\rm is~odd},~{\rm and}~j\neq t,\\
n\lfloor{2m+2\over5}\rfloor &   ~{\rm otherwise}.
\end{array}\right.$
\end{description}where $n_1, n_2,\ldots,n_j$ are at least $2$, $n_{j+1}=\ldots=n_t=1$, and
$\sum_{i=1}^tn_i=n$.
\end{cor}
\begin{cor} For $n\geq 2$, \begin{description}\item (1)
$\dim({K_n[K_m]})=nm-1$ \item (2) $\dim({P_n[K_m]})=\left\{
\begin{array}{ll}
n(m-1) &  ~{\rm if}~n\geq3, \\
n(m-1)+1 &  ~{\rm if}~n=2. \end{array}\right.$ \item (3)
$\dim({C_n[K_m]})=\left\{
\begin{array}{ll}
n(m-1) &  ~{\rm if}~n\geq4, \\
n(m-1)+2 &  ~{\rm if}~n=3. \end{array}\right.$ \item (4)
$\dim({K_{n_1,n_2,\ldots,n_t}[K_m]})=\left\{
\begin{array}{ll}
n(m-1)+t-j-1 &  ~{\rm if}~j\neq t, \\
n(m-1) &   ~{\rm if}~j=t,
\end{array}\right.$\\where $n_1, n_2,\ldots,n_j$ are at least $2$, $n_{j+1}=\ldots=n_t=1$, and
$\sum_{i=1}^tn_i=n$. \item (5) $\dim({K_n[\overline
K_m]})=n(m-1)$ \item (6) $\dim({P_n[\overline K_m]})=\left\{
\begin{array}{ll}
n(m-1) &  ~{\rm if}~n\neq3, \\
n(m-1)+1 &  ~{\rm if}~n=3. \end{array}\right.$ \item (7)
$\dim({C_n[\overline K_m]})=\left\{
\begin{array}{ll}
n(m-1) &  ~{\rm if}~n\neq4, \\
n(m-1)+2 &  ~{\rm if}~n=4. \end{array}\right.$ \item (8)
$\dim({K_{n_1,n_2,\ldots,n_t}[\overline K_m]})=n(m-1)+n-t$, where
$n_1, n_2,\ldots,n_j$ are at least $2$, $n_{j+1}=\ldots=n_t=1$, and
$\sum_{i=1}^tn_i=n$.
\end{description}
\end{cor}
\begin{cor} Let  $m_1,\ldots,m_q\geq2$,
$m_{q+1}=\ldots=m_s$, and $m=\sum_{i=1}^sm_i$. Then for $n\geq 2$,
\begin{description} \item (1) $\dim({K_n[K_{m_1,\ldots,m_s}]})=\left\{
\begin{array}{ll}
n(m-q) &  ~{\rm if}~q=s, \\
n(m-q)-1 &  ~{\rm otherwise}. \end{array}\right.$ \item (2)
$\dim({P_n[K_{m_1,\ldots,m_s}]})=\left\{
\begin{array}{ll}
n(m-q) &  ~{\rm if}~q=s, \\
n(m-q-1) &  ~{\rm if}~q\neq s ~{\rm and} ~n\geq3, \\
n(m-q-1)+1 &  ~{\rm otherwise}. \end{array}\right.$ \item (3)
$\dim({C_n[K_{m_1,\ldots,m_s}]})=\left\{
\begin{array}{ll}
n(m-q) &  ~{\rm if}~q=s, \\
n(m-q-1) &  ~{\rm if}~q\neq s ~{\rm and}~ n\geq4, \\
n(m-q-1)+2 &  ~{\rm otherwise}. \end{array}\right.$ \item (4)
$\dim({K_{n_1,n_2,\ldots,n_t}[K_{m_1,\ldots,m_s}]})=\left\{
\begin{array}{ll}
n(m-q) &  ~{\rm if}~q=s, \\
n(m-q-1) &  ~{\rm if}~q\neq s ~{\rm and}~ j=t, \\
n(m-q-1)+t-j-1 &  ~{\rm otherwise}, \end{array}\right.$\\
where $n_1, n_2,\ldots,n_j$ are at least $2$, $n_{j+1}=\ldots=n_t=1$,
and $\sum_{i=1}^tn_i=n$.
 \item (5) $\dim({K_n[\overline
K_{m_1,\ldots,m_s}]})=\left\{
\begin{array}{ll}
n(m-q) &  ~{\rm if}~q=s, \\
n(m-q-1) & ~{\rm otherwise}.
\end{array}\right.$ \item (6) $\dim({P_n[\overline
K_{m_1,\ldots,m_s}]})=\left\{
\begin{array}{ll}
n(m-q) &  ~{\rm if}~q=s, \\
n(m-q-1) &  ~{\rm if}~q\neq s ~{\rm and}~ n\neq3, \\
n(m-q-1)+1 &  ~{\rm otherwise}. \end{array}\right.$ \item (7)
$\dim({C_n[\overline K_{m_1,\ldots,m_s}]})=\left\{
\begin{array}{ll}
n(m-q) &  ~{\rm if}~q=s, \\
n(m-q-1) &  ~{\rm if}~q\neq s ~{\rm and}~ n\neq4, \\
n(m-q-1)+2 &  ~{\rm otherwise}. \end{array}\right.$ \item (8)
$\dim({K_{n_1,n_2,\ldots,n_t}[\overline K_{m_1,\ldots,m_s}]})=\left\{
\begin{array}{ll}
n(m-q) &  ~{\rm if}~q=s, \\
n(m-q)-t &  ~{\rm otherwise},
\end{array}\right.$\\where $\sum_{i=1}^tn_i=n$.
\end{description}
\end{cor}

\begin{cor} Let $P$ be the Petersen graph. Then for $n\geq 2$, \begin{description}
\item (1) $\dim({K_n[P]})=3n$ \item (2) $\dim({P_n[P]})=\left\{
\begin{array}{ll}
3n &  ~{\rm if}~n\neq3, \\
3n+1 &  ~{\rm if}~n=3. \end{array}\right.$ \item (3)
$\dim({C_n[P]})=\left\{
\begin{array}{ll}
3n &  ~{\rm if}~n\neq4, \\
3n+2 &  ~{\rm if}~n=4. \end{array}\right.$ \item (4)
$\dim({K_{n_1,n_2,\ldots,n_t}[P]})=4n-t$, where $\sum_{i=1}^tn_i=n$.
 \item (5) $\dim({K_n[\overline
P]})=4n-1$ \item (6) $\dim({P_n[\overline P]})=\left\{
\begin{array}{ll}
3n &  ~{\rm if}~n\geq3, \\
7 &  ~{\rm if}~n=2. \end{array}\right.$ \item (7)
$\dim({C_n[\overline P]})=\left\{
\begin{array}{ll}
3n &  ~{\rm if}~n\geq4, \\
11 &  ~{\rm if}~n=3. \end{array}\right.$ \item (8)
$\dim({K_{n_1,n_2,\ldots,n_t}[\overline P]})=\left\{
\begin{array}{ll}
3n+t-j-1 &  ~{\rm if}~j\neq t, \\
3n &   ~{\rm if}~j=t,
\end{array}\right.$\\where $n_1, n_2,\ldots,n_j$ are at least $2$, $n_{j+1}=\ldots=n_t=1$, and
$\sum_{i=1}^tn_i=n$.
\end{description}
\end{cor}

 It is easy to see that if $G$ is a bipartite
graph of order at least $3$, then it does not have any pair of
adjacent twins. Therefore, by Theorem~\ref{nB2+b(G) N(G)},
$\dim(G[H])=n\dim_2(H)+b(G)-\iota_{_N}(G)$.



\begin{thebibliography}{99}

\bibitem{baily}
{ R.F. Bailey and P.J. Cameron}, {\it Base size, metric dimension
and other invariants of groups and graphs},~Bull. London Math.
Soc. {\bf  43} (2011) 209-242.

\bibitem{net2}
{Z. Beerliova, F. Eberhard, T. Erlebach, A. Hall, M. Hoffmann, M. Mihalak and L. S. Ram},
Network dicovery and verification,
{\em IEEE Journal On Selected Areas in Communications} {\bf24(12)} (2006) 2168-2181.

\bibitem{K dimensional} { P.S. Buczkowski, G. Chartrand, C. Poisson, and P. Zhang}, {\it On k-dimensional
graphs and their bases},~Periodica Mathematica Hungarica {\bf
46(1)} (2003) 9-15.

\bibitem{cartesian product}
 { J. Caceres, C. Hernando, M. Mora, I.M. Pelayo, M.L. Puertas, C. Seara, and D.R. Wood},
{\it On the metric dimension of cartesian products of graphs},~SIAM Journal Discrete Mathematics
{\bf  21(2)} (2007) 423-441.


\bibitem{Ollerman} { G. Chartrand, L. Eroh, M.A. Johnson, and O.R. Ollermann},
 {\it Resolvability in graphs and the metric dimension of a graph},~Discrete Applied Mathematics {\bf  105} (2000) 99-113.

\bibitem{Harary}{ F. Harary and R.A Melter}, {\it On the metric dimension of a graph},~Ars Combinatoria {\bf  2} (1976) 191-195.

\bibitem{extermal} { C. Hernando, M. Mora, I.M. Pelayo, C. Seara, and D.R. Wood},
{\it Extremal Graph Theory for Metric Dimension and Diameter},~
The Electronic Journal of Combinatorics {\bf 17 } (2010) \#R30.


\bibitem{lexico}{ M. Jannesari  and B. Omoomi},
 {\it The metric dimension of the  lexicographic product of
graphs},~Discrete Mathematics.
{\bf 312(22)} (2012) 3349-3356.

\bibitem{landmarks}{ S. Khuller, B. Raghavachari,  and A. Rosenfeld},
 {\it Landmarks in graphs},~Discrete Applied Mathematics
{\bf  70(3)} (1996) 217-229.

\bibitem{digital}{R. A. Melter and I. Tomescu},
{Metric bases in digital geometry},
{\em Computer Vision Graphics and Image Processing} {\bf 25} (1984) 113-121.

\bibitem{coin}{A. Sebo and E. Tannier},
{On metric generators of graphs},
{\em Mathematics of Operations Research} {\bf 29(2)} (2004) 383-393.

\bibitem{Slater1975} { P.J. Slater}, {\it Leaves of trees},
Congressus Numerantium {\bf  14} (1975) 549-559.

  \end{thebibliography}
\end{document}